\documentclass{article}
\usepackage{amsthm}
\usepackage{amsmath}
\usepackage{amssymb}
\usepackage{graphicx} 
\usepackage{indentfirst}
\usepackage{color}

\newtheorem{theorem}{Theorem}[section]
\newtheorem{definition}[theorem]{Definition}
\newtheorem{lemma}[theorem]{Lemma}
\newtheorem{proposition}[theorem]{Proposition}

\newtheorem{corollary}[theorem]{Corollary}
\newtheorem{problem}[theorem]{Problem}
\newtheorem{example}[theorem]{Example}

\newcommand{\ord}{\mathrm{ord}}
\newcommand{\supp}{\mathrm{supp}}
\newcommand{\rad}{\mathrm{rad}}

\makeatletter
\newcommand{\blfootnote}[1]{%
  \begingroup
  \renewcommand{\thefootnote}{}%
  \renewcommand{\@makefntext}[1]{\noindent ##1}%
  \footnotetext{#1}%
  \addtocounter{footnote}{-1}%
  \endgroup
}
\makeatother

\numberwithin{equation}{section}

\title{Two local zero-sum problems}
\author{
Weidong Gao \\
Center for Applied Mathematics, Tianjin University \\
Tianjin 300072, P.R. China \\
E-mail: \texttt{wdgao@nankai.edu.cn}
\and
Xiao Jiang \\
School of Mathematical Sciences, Chengdu University of Technology \\
Chengdu 610059, P.R. China \\
E-mail: \texttt{jiangxiao2022@tju.edu.cn}
\and
Yucen Mu \\
Center for Applied Mathematics, Tianjin University \\
Tianjin 300072, P.R. China \\
E-mail: \texttt{yucen4617@tju.edu.cn}
}

\begin{document}

\maketitle

\begin{abstract}
    In the present paper, we investigate two local zero-sum problems. Let $n,k\ge 2$. We denote by $\mathsf{D}^*(n,nk)$ (resp. $\eta^{*}(n,nk)$) the smallest positive integer $\ell$ (if exists) such that, from any given $\ell$ integers not divisible by $n$, one can select some (resp. at most $n$) of them whose sum is divisible by $n$ but not by $nk$.
    We prove that both $\mathsf{D}^*(n,nk)$ and $\eta^{*}(n,nk)$ are equal to $2n-1$ if $\rad(n)\mid\rad(k)$ and infinite otherwise. The corresponding inverse problem is also determined.

    We denote by $\mathsf{D}_n^{\times}$ (resp. $\eta_n^{\times}$) the smallest positive integer $\ell$ such that, from any given $\ell$ integers coprime to $n$, one can select some (resp. at most $n$) of them whose sum $\sigma$ satisfies $\gcd(\sigma, n^2)=n$.
    We prove that $\mathsf{D}_n^{\times}=\eta_n^{\times}=2n-1$ if $n$ is a prime power, and determine its inverse problem.
\end{abstract}

\blfootnote{
\textbf{2020 Mathematics Subject Classification: 11B30, 11B75, 20K01.} \\
\textbf{Keywords: Finite abelian group, Local zero-sum problems, Zero-sum theory.}
}



\section{Introduction and main results}

The zero-sum theory in finite abelian groups is a central topic in additive combinatorics and combinatorial number theory. A basic direct zero-sum problem asks for conditions guaranteeing that every sufficiently long sequence over a finite abelian group contains a non-empty zero-sum subsequence with prescribed properties, while the corresponding inverse problem studies the structure of extremal sequences for which such subsequences do not exist. For a recent systematic survey on the zero-sum problem, we refer to \cite{[Gao2006]}.

This line of research originated in the celebrated theorem of~Erd\H{o}s, Ginzburg and Ziv (see \cite{[EGZ1961]}), which asserts that out of any $2n-1$ integers, one can select exactly $n$ integers whose sum is a multiple of $n$. Soon afterwards, Davenport's questions led to the introduction of the Davenport constant (see \cite{[Olson1 1969], [Olson2 1969], [PD1967]}), and these two directions became the starting point for a broad and influential theory.

Over the years, several classical invariants have been introduced to measure the existence of zero-sum subsequences under various length restrictions, among them the Davenport constant $\mathsf{D}(G)$, the invariant $\eta(G)$, and the Erd\H{o}s--Ginzburg--Ziv constant $\mathsf{s}(G)$ (see \cite{[Gao2006]} and Definition \ref{Definition 2.1}). These invariants are closely connected with inverse problems, additive number theory, and the arithmetic of Krull monoids. For more information, we refer to \cite{[GH2006]}.

Recently, an interesting local zero-sum problem has arisen in the study of certain zero-sum invariants related to the cross number (see \cite{[GHJLW2025]}).
For a prime $p$, $\mathsf{s}_p^*$ is defined as the smallest integer $\ell$ such that every $\ell$ integers coprime to $p$ contain $p$ terms whose sum is divisible by $p$ but not by $p^2$.
Determining $\mathsf{s}_p^*$ may be viewed as a refinement of the Erd\H{o}s--Ginzburg--Ziv Theorem in which one asks not only for divisibility by $p$, but also for a more precise control of the $p$-adic valuation of the resulting sum.
It has been proved that $\mathsf{s}_2^*=3$ and $2p+1\le \mathsf{s}_p^* \le 3p-2$ for an odd prime $p$ (see \cite{[GJLLY2026]}).
And it is conjectured that $\mathsf{s}_p^*=2p+1$ for any odd prime $p$.

Motivated by this circle of ideas, we investigate two local zero-sum problems related to $\mathsf{s}_p^*$ and their inverse problems. For convenience, we introduce some new notations.

\begin{definition}\label{Definition 1.1}
    For integers $n, k\ge 2$, we denote by

    \ensuremath{\bullet} $\mathsf{D}^*(n,nk)$ the smallest positive integer $\ell$ (if exists) such that, from any given $\ell$ integers not divisible by $n$, one can select some of them whose sum is divisible by $n$ but not by $nk$. If no such $\ell$ exists, then we set $\mathsf{D}^*(n,nk)=\infty$.

    \ensuremath{\bullet} $\eta^*(n,nk)$ the smallest positive integer $\ell$ (if exists) such that, from any given $\ell$ integers not divisible by $n$, one can select at most $n$ of them whose sum is divisible by $n$ but not by $nk$. If no such $\ell$ exists, then we set $\eta^*(n,nk)=\infty$.

    \ensuremath{\bullet} $\mathsf{D}^{\times}(n,nk)$ the smallest positive integer $\ell$ such that, from any given $\ell$ integers coprime to $nk$, one can select some integers whose sum $\sigma$ satisfies $\gcd(\sigma, nk)=n$.

    \ensuremath{\bullet} $\eta^{\times}(n,nk)$ the smallest positive integer $\ell$ such that, from any given $\ell$ integers coprime to $nk$, one can select at most $n$ of them whose sum $\sigma$ satisfies $\gcd(\sigma, nk)=n$.

    If $k=n$, then we simply denote
    $\mathsf{D}_{n}^{*}=\mathsf{D}^*(n,n^2)$,
    $\eta_{n}^{*}=\eta^*(n,n^2)$, $\mathsf{D}_{n}^{\times}=\mathsf{D}^{\times}(n,n^2)$
    and $\eta_{n}^{\times}=\eta^{\times}(n,n^2)$.
\end{definition}

The first local zero-sum problem is to determine the values of $\mathsf{D}^*(n,nk)$ and $\eta^*(n,nk)$ for integers $n, k\ge 2$.
Very recently, Sun determined $\mathsf{D}_n^*=2n-1$, as well as the structure of sequences in the corresponding inverse problem (see \cite{[S2026+]}). In the present paper, we have completely solved this problem and its inverse problem. 
We state our results as follows.

\begin{theorem}\label{Theorem 1.2}
    Let $n, k\ge 2$ be integers.

    {\rm (1)} If $\rad(n) \nmid \rad(k)$,
    then $\mathsf{D}^*(n,nk)=\eta^*(n,nk)=\infty$.

    {\rm (2)} If $\rad(n) \mid \rad(k)$,
    then $\mathsf{D}^*(n,nk)=\eta^*(n,nk)=2n-1$.
\end{theorem}

\begin{theorem}\label{Theorem 1.3}
    Let $n, k, m$ be positive integers with
    $m \mid k$, $\gcd(m, k/m)=1$ and $\rad(k/m)=\rad(n)>1$.
    Let $S$ be a sequence of $2n-2$ integers, none of which is divisible by $n$. Then $S$ has no subsequence of length at most $n$ whose sum is divisible by $n$ but not by $nk$ if and only if
    there exist integers $c,d$ with $c$ coprime to $n$ such that $n-1$ terms of $S$ are congruent to $d\frac{nk}{m}+cm$ modulo $nk$ and the other $n-1$ terms of $S$ are congruent to $-d\frac{nk}{m}-cm$ modulo $nk$.
    In fact, such a sequence $S$ has no subsequence whose sum is divisible by $n$ but not by $nk$.
\end{theorem}

The second local zero-sum problem is about $\mathsf{D}^{\times}(n,nk)$ and $\eta^{\times}(n,nk)$.
It is clear that $\mathsf{D}^{\times}(n,nk)\le \eta^{\times}(n,nk)<\infty$
for any integers $n, k\ge 2$.
And if $k$ and $n$ are powers of the same prime,
then $\mathsf{D}_{n}^{\times}=\mathsf{D}^{\times}(n,nk)$
and $\eta_{n}^{\times}=\eta^{\times}(n,nk)$.
We have solved this problem and its inverse problem in this case, 
and obtain the following two results.

\begin{theorem}\label{Theorem 1.4}
     If $n$ is a prime power, then $\mathsf{D}_{n}^{\times}=\eta_{n}^{\times}=2n-1$.
\end{theorem}

\begin{theorem}\label{Theorem 1.5}
    Let $n$ be a power of a prime $p$.
    Let $S$ be a sequence of $2n-2$ integers coprime to $n$.
    Then $S$ has no subsequence of length at most $n$ whose sum $\sigma$ satisfies $\gcd(\sigma, n^2)=n$ if and only if there exists an integer $c$ coprime to $n$ such that $n-1$ terms of $S$ are congruent to $c$ modulo $np$ and the other $n-1$ terms of $S$ are congruent to $-c$ modulo $np$.
    In fact, such a sequence $S$ has no subsequence whose sum $\sigma$ satisfies $\gcd(\sigma, n^2)=n$.
\end{theorem}

The paper is organized in the following way.
Some notations and some preliminary results are gathered in Section 2.
Section 3 provides some essential lemmas,
including the prime case $n=p$ of Theorem \ref{Theorem 1.2} and Theorem \ref{Theorem 1.4}.
Then in Section 4, we prove Theorem \ref{Theorem 1.2} and Theorem \ref{Theorem 1.4}.
In Section 5, we focus on the inverse problem and present the proofs of Theorem \ref{Theorem 1.3} and Theorem \ref{Theorem 1.5}.
Section 6 is devoted to some remarks.

\section{Notation and preliminaries}

We denote by $\mathbb{N}$ the set of positive integers, by $\mathbb{N}_0$ the set $\mathbb{N} \cup \{0\}$,  by $\mathbb{Z}$ the set of integers.
For any $a,b \in \mathbb{Z}$, we set $[a,b]=\{x\in \mathbb{Z}\mid a\le x\le b\}$.
For any $n \in \mathbb{N}$, $\rad(n)=\prod_{p \mid n} p$ is the product of all distinct prime divisors of $n$ with $\rad(1)=1$.
For any real number $x$, let $\lceil x \rceil$ be the smallest integer not less than $x$, and $\lfloor x \rfloor$ be the largest integer no greater than $x$.

Throughout this paper, let $p$ be a prime and $G$ be a finite additive abelian group with identity element $0$.
We denote by $C_n$ a cyclic group of order $n$.
By the structure of finite abelian groups, every non-trivial finite abelian group $G$
can be decomposed as a direct sum of cyclic groups
$$
C_{n_1} \oplus \cdots \oplus C_{n_r},
$$
where the integers $n_1 , \cdots , n_r$ satisfy $ 1 < n_1 \mid \cdots \mid n_r $.
Here $r$ and $n_r$ are called the \textit{rank} and the \textit{exponent} of $G$, respectively.
The exponent of $G$ is denoted by $\exp(G)$,
and the order of an element $g \in G$ is written as $\ord(g)$.
Set
$$
G^{\times}=\{g\in G:\ord(g)=\exp(G)\}.
$$

For a set $X$, we denote by $\mathcal{F}(X)$ the free (abelian, multiplicative)
monoid with basis $X$. The elements of $\mathcal{F}(X)$ are called \emph{sequences} over $X$.
We can write a sequence $S\in \mathcal{F}(X)$ in the form
$$
S=\prod_{g\in X}g^{\mathsf{v}_g(S)},
$$
where $\mathsf{v}_g(S)\in \mathbb{N}_0$ is called the \emph{multiplicity} of the element $g$ in $S$.
We say that $S$ \emph{contains} the element $g$ of $X$ if $\mathsf{v}_g(S)>0$.
A sequence $S'$ is called a \emph{subsequence} of $S$, denoted by $S'\mid S$, if $\mathsf{v}_g(S')\le \mathsf{v}_g(S)$ for all $g\in X$.
The sequence $SS'^{-1}$ denotes the subsequence obtained from $S$ by deleting $S'$. Two subsequences $S_1$ and $S_2$ of $S$ are called \emph{disjoint} if $S_1 | SS_2^{-1}$. The identity element $\emptyset \in \mathcal{F}(X)$ is called the \emph{empty} sequence.

For a group $G$ and a sequence
$$S=g_1\cdots g_l=\prod_{g\in G}g^{\mathsf{v}_g(S)}\in \mathcal{F}(G),$$
we define
\begin{itemize}
\item $|S| = l =\sum_{g\in G}\mathsf{v}_{g}(S)\in \mathbb{N}_{0}$
the \emph{length} of $S$,

\item $\mathsf{h}(S) = \max\{\mathsf{v}_g(S)\ |\ g \in G\}\in[0, |S|]$ the \emph{maximum of the multiplicities} of $S$,

\item $\sigma(S)=\sum_{i=1}^{l}g_{i}=\sum_{g\in G}\mathsf{v}_{g}(S)g
\in G$ the \emph{sum} of $S$,

\item $\supp(S)=\{g\in G\mid \mathsf{v}_g(S)>0\}\subseteq G$ the \emph{support} of $S$,

\item $\Sigma(S)=\{\sigma(T): \ T\mid S,\ |T|\ge 1\}$
the set of sums of all non-empty subsequences of $S$,

\item $\Sigma_{k}(S)=\{\sigma(T): \ T\mid S,\ |T|=k\}$
for all $k\in \mathbb{N}$,

\item $\Sigma_{\le k}(S)=\{\sigma(T): \ T\mid S,\ 1\le |T|\le k\}$
for all $k\in \mathbb{N}$.

\item $S$ a \emph{squarefree sequence} if $\mathsf{v}_{g}(S)\le 1$
for all $g\in G$,

\item $S$ a \emph{zero-sum sequence} if $\sigma(S)=0\in G$,

\item $S$ a \emph{zero-sum free sequence} if $0\not\in \Sigma(S)$,
\end{itemize}

Every map of abelian groups $\varphi:G\longrightarrow H$ can extend to a homomorphism $\varphi:\mathcal{F}(G)\longrightarrow \mathcal{F}(H)$,
where $\varphi(S)=\varphi(g_1)\cdots \varphi(g_l)$.
If $\varphi$ is a homomorphism,
then $\varphi(S)$ is a zero-sum sequence over $H$
if and only if $\sigma(S)\in \ker(\varphi)$.
Moreover, $\varphi( \sigma(S) )=\sigma( \varphi(S) )$.

We collect some known results which will be used later.

\begin{definition}\label{Definition 2.1}
\cite{[Gao2006]}
    Let $G$ be a finite abelian group of exponent $\exp(G)$.
    We denote by

    \ensuremath{\bullet} $\mathsf{D}(G)$ the smallest integer $\ell$ such that every sequence $S\in\mathcal{F}(G)$ of length $|S|\ge \ell$ has a non-empty zero-sum subsequence.

    \ensuremath{\bullet} $\eta(G)$ the smallest integer $\ell$ such that every sequence $S\in\mathcal{F}(G)$ of length $|S|\ge \ell$ has a zero-sum subsequence $T$ of length $1\le|T|\le \exp(G)$.

    \ensuremath{\bullet} $\mathsf{s}(G)$ the smallest integer $\ell$ such that every sequence $S\in\mathcal F(G)$ of length $|S|\ge \ell$ has a zero-sum subsequence $T$ of length $\exp(G)$.
\end{definition}

\begin{lemma}\label{Lemma 2.2}
{\cite[Theorem 5.8.3]{[GH2006]}}
    Let $G=C_{n_1}\oplus C_{n_2}$ with $1\le n_1 \mid n_2$. Then
    $$\eta(G)=2n_1+n_2-2 \quad \text{and}\quad \mathsf{s}(G)=2n_1+2n_2-3.
    $$
\end{lemma}

\begin{lemma}\label{Lemma 2.3}
\cite{[D1935]}
    Let $p$ be a prime, and let $A_1, A_2, \ldots, A_k$ be non-empty subsets of $C_p$. Then
    $$
    |A_1+A_2+\cdots +A_k| \ge \min \{p, |A_1|+|A_2|+\cdots +|A_k|-k+1 \}.
    $$
\end{lemma}

\begin{lemma}\label{Lemma 2.4}
\cite{[C1935], [H1990]}
    Let $A,B\subseteq C_n$ with $n\ge 2$ be non-empty subsets.
    If $0\in B$ and $B\setminus\{0\} \subseteq C_n^{\times}$, then
    $$|A+B|\ge \min\{n,\ |A|+|B|-1\}.$$
\end{lemma}

\begin{corollary}\label{Corollary 2.5}
    Let $A_1,\ldots,A_{\ell}$ be subsets of $C_n$ with integers $\ell, n\ge 2$.
    If $0\in A_i$ and $A_i\setminus\{0\} \subseteq C_n^{\times}$ for every $i\in[1, \ell]$, then
    $$|A_1+\cdots+A_{\ell}|\ge
    \min \left\{n,\ \sum_{i=1}^{\ell} |A_i|-\ell+1 \right\}.$$
\end{corollary}

\begin{proof}
    This result follows from Lemma \ref{Lemma 2.4} by induction on $\ell$.
\end{proof}

\begin{lemma}\label{Lemma 2.6}
\cite[Lemma 2.4]{[WT2024]}
    Let $n\ge 2$ and let $g\in C_n^{\times}.$
    Let $A$ and $B$ be non-empty subsets of $C_n$ such that $\min\{|A|,|B|\}\ge 2$
    and
    $$|A+B|=|A|+|B|-1.$$
    If $A$ is an arithmetic progression with common difference $g$, then $B$ is an
    arithmetic progression with the same common difference $g$.
\end{lemma}

\section{Some essential lemmas}
In this section, we provide several essential lemmas that will be used to prove our main results.

By Definition \ref{Definition 1.1}, it is clear that
$\eta^*(n,nk) \ge \mathsf{D}^*(n,nk)$,
$\eta^{\times}(n,nk) \ge \mathsf{D}^{\times}(n,nk)$, $\mathsf{D}_{p}^{*}=\mathsf{D}_{p}^{\times}$
and $\eta_{p}^{*}=\eta_{p}^{\times}$.
And the following proposition helps us handle them in certain cyclic groups
so that many classical methods in the zero-sum theory can be applied.

\begin{proposition}\label{Proposition 3.1}
    Let $G$ be a cyclic group with $nk \mid |G|$ and $n, k\ge 2$.

    {\rm (1)} $\mathsf{D}^{*}(n,nk)$ (resp. $\eta^{*}(n,nk)$) is the smallest integer $\ell$ such that every sequence $S\in\mathcal{F}(G\setminus nG)$ of length $| S | \ge \ell$ has a non-empty subsequence $T$ (resp. of length $1\le|T|\le n$) satisfying $\sigma(T)\in nG\setminus nkG$.

    {\rm (2)} If $\rad(k)=\rad(\frac{|G|}{n})$, then $\mathsf{D}^{\times}(n,nk)$ (resp. $\eta^{\times}(n,nk)$) is the smallest integer $\ell$ such that every sequence $S\in\mathcal{F}(G^{\times})$ of length $| S | \ge \ell$ has a subsequence $T$ (resp. of length $1\le|T|\le n$) satisfying $\sigma(T)\in (nG)^{\times}$.
\end{proposition}

First, we establish the equivalence between $\eta^*(n,nk)$ and $\mathsf{D}^*(n,nk)$,
as well as $\eta^{\times}(n,nk)$ and $\mathsf{D}^{\times}(n,nk)$.
This equivalence will significantly simplify the subsequent evaluations.

\begin{lemma}\label{Lemma 3.2}
    Let $n, k\ge 2$ be integers.

    {\rm (1)} If $\mathsf{D}^*(n,nk)< \infty$, we have $\eta^*(n,nk)=\mathsf{D}^*(n,nk)$.

    {\rm (2)} If both $k$ and $n$ are powers of the same prime,
    then $\eta^{\times}(n,nk)=\mathsf{D}^{\times}(n,nk)$.
\end{lemma}

\begin{proof}
    {\rm (1)} We only need to prove $\eta^*(n,nk) \le \mathsf{D}^*(n,nk)$.
    Let $S\in\mathcal{F}(C_{nk} \setminus nC_{nk})$ with $| S | =\mathsf{D}^*(n,nk)$.
    We are going to prove that $S$ has a subsequence $S'$ with $1 \le | S'| \le n$ and $\sigma(S')\in nC_{nk}\setminus\{0\}$.

    By the definition of $\mathsf{D}^*(n,nk)$, we can select a subsequence $T$ of $S$ satisfying
    \begin{equation}\label{3-1}
        \sigma(T)\in nC_{nk}\setminus nkC_{nk}=nC_{nk}\setminus\{0\}.
    \end{equation}
    If $1 \le |T| \le n$, then $T$ is as desired. Thus, we just need to consider the case $|T| \ge n+1$.

    Let $\pi : C_{nk} \rightarrow C_{nk}/nC_{nk}\cong C_n$ be the canonical homomorphism. Since $\eta (C_n)=n$ by Lemma \ref{Lemma 2.2}, there exists a subsequence $T_1$ of $T$ with $1 \le |T_1| \le n$ and $\sigma(\pi(T_1))=0+nC_{nk}\in C_{nk}/nC_{nk}$.
    It follows from $\pi(\sigma(T_1))=\sigma(\pi(T_1))$ that
    $$\sigma(T_1)\in \ker(\pi)=nC_{nk}.$$

    If $\sigma(T_1)\in nC_{nk}\setminus\{0\}$, then $T_1$ is as desired.
    If $\sigma(T_1)=0$, then let us consider the sequence $T_1^{-1}T$.
    By (\ref{3-1}), we have
    $$\sigma(T_1^{-1}T)=\sigma(T)-\sigma(T_1)\in nC_{nk}\setminus\{0\}$$
    and $1\le |T_1^{-1}T|<|T|$.

    If $1 \le |T_1^{-1}T| \le n$, then $T_1^{-1}T$ is as desired.
    If $|T_1^{-1}T| \ge n+1$, then by repeating the above operation,
    we must obtain a subsequence of $S$ as desired
    since the length of $S$ is finite.

    {\rm (2)} We only need to prove $\eta^{\times}(n,nk) \le \mathsf{D}^{\times}(n,nk)$.
    Let $S\in\mathcal{F}(C_{nk}^{\times})$ with $|S| =\mathsf{D}^{\times}(n,nk)$.
    We are going to prove that $S$ has a subsequence $S'$ with $1 \le |S'| \le n$ and $\sigma(S')\in (nC_{nk})^{\times}$.

    By the definition of $\mathsf{D}^{\times}(n,nk)$, we can select a subsequence $T$ of $S$ satisfying $\sigma(T)\in (nC_{nk})^{\times}$, or equivalently,
    \begin{equation}\label{3-2}
        \ord(\sigma(T))=k=|nC_{nk}|.
    \end{equation}
    If $1 \le |T| \le n$, then $T$ is as desired. Thus, we just need to consider the case $|T| \ge n+1$.

    Let $\pi : C_{nk} \rightarrow C_{nk}/nC_{nk}\cong C_n$ be the canonical homomorphism. Since $\eta (C_n)=n$, there exists a subsequence $T_1$ of $T$ with $1 \le |T_1| \le n$ and $\sigma(\pi(T_1))=0+nC_{nk}\in C_{nk}/nC_{nk}$.
    It follows from $\pi(\sigma(T_1))=\sigma(\pi(T_1))$ that
    $$\sigma(T_1)\in \ker(\pi)=nC_{nk}.$$

    If $\sigma(T_1)\in (nC_{nk})^{\times}$, then $T_1$ is as desired.
    If $\sigma(T_1)\in nC_{nk} \setminus (nC_{nk})^{\times}$, then by (\ref{3-2}), we have $\ord(\sigma(T_1))\mid |nC_{nk}|=\ord(\sigma(T))$ and $\ord(\sigma(T_1)) \neq \ord(\sigma(T))$, which imply that
    \begin{equation}\label{3-3}
       \ord(\sigma(T)-\sigma(T_1))=\ord(\sigma(T))=|nC_{nk}|
    \end{equation}
    since $C_{nk}$ is a cyclic $p$-group.

    Let us consider the sequence $T_1^{-1}T$.
    If $1 \le |T_1^{-1}T| \le n$, then $T_1^{-1}T$ is as desired since $\ord(\sigma(T_1^{-1}T))=\ord(\sigma(T)-\sigma(T_1))=|nC_{nk}|$ by (\ref{3-3}).
    If $|T_1^{-1}T| \ge n+1$, then by repeating the above operation,
    we must obtain a subsequence of $S$ as desired
    since the length of $S$ is finite.
\end{proof}

By constructing a specific sequence, the following lemma provides a lower bound for $\mathsf{D}^*(n,nk)$ and $\mathsf{D}^{\times}(n,nk)$.

\begin{lemma}\label{Lemma 3.3}
    Let $n, k\ge 2$ be integers and $g\in C_{nk}^{\times}$. Set
    \begin{equation*}
        S=g^{n-1}(-g)^{n-1}.
    \end{equation*}
    Then $S$ has no subsequence $T$ such that $\sigma(T)\in nC_{nk}\setminus\{0\}.$
\end{lemma}

\begin{proof}
    For any non-empty subsequence $T \mid S$, $T$ can be written as $T=g^a(-g)^b$,
    where $0\le a, b\le n-1$ and $a+b\ge 1$.
    If $\sigma(T)=(a-b)g \in nC_{nk}$, then we have $a-b\equiv 0 \pmod n$. Since $0\le a,b\le n-1$, this congruence forces $a=b$. Thus, we must have $\sigma(T)=0\in nC_{nk}$.
\end{proof}

From the above lemmas, we get
\begin{equation*}
    \eta^*(n,nk) = \mathsf{D}^*(n,nk) \ge 2n-1
    \text{~~and~~}
    \eta^{\times}(n,nk) \ge \mathsf{D}^{\times}(n,nk) \ge 2n-1.
\end{equation*}

\begin{lemma}\label{Lemma 3.4}
    Let $S\in\mathcal{F}(C_{nk}^{\times})$ with integers $n, k\ge 2$
    be a sequence of length $|S|\ge n+1$.
    If there exist two elements of $S$ with their difference in $nC_{nk}\setminus \{0\}$,
    then $S$ has a subsequence $T$ satisfying $1\le |T| \le n$ and $\sigma(T)\in nC_{nk}\setminus \{0\}$.
\end{lemma}

\begin{proof}
    Without loss of generality, let $S=g_1g_2\cdots g_{n}g_{n+1}$ with
    $$
    g_{n+1}-g_{n}\in nC_{nk}\setminus \{0\}.
    $$
    Let $\pi : C_{nk} \rightarrow C_{nk}/nC_{nk}\cong C_n$ be the canonical homomorphism.
    Set $A_{i}=\{\pi(0), \pi(g_i)\}\subseteq C_{nk}/nC_{nk}$ for each $i\in[1, n-1]$.
    By Corollary \ref{Corollary 2.5}, we obtain
    \begin{equation*}
        | A_1+\cdots+A_{n-1} | \ge \min\{n, 2(n-1)-(n-1)+1\}=n,
    \end{equation*}
    which implies
    $$-\pi(g_n) \in C_{nk}/nC_{nk}=A_1+\cdots+A_{n-1}.$$
    Therefore, there is a subsequence $U$ of $g_1g_2\cdots g_{n-1}$
    with $0 \le |U| \le n-1$ such that
    $\pi(\sigma(g_nU))=\pi(g_n)+\sigma(\pi(U))=\pi(0)=0+nC_{nk}$.
    Since $g_{n+1}-g_{n}\in nC_{nk}\setminus \{0\}$, we have
    $$\sigma(g_nU)\in nC_{nk}, \sigma(g_{n+1}U)\in nC_{nk}
    \text{~~and~~}\sigma(g_{n}U)\neq \sigma(g_{n+1}U).$$
    Thus, either $\sigma(g_{n}U)$ or $\sigma(g_{n+1}U)$ is in $nC_{nk}\setminus \{0\}$.
\end{proof}

\begin{lemma}\label{Lemma 3.5}
    For an integer $n\ge 2$, let $S \in \mathcal{F}(C_{n}^{\times})$ and $h\in [1, n]$. If $| S | \ge n-1$ and $\mathsf{h}(S)\le h$, then $C_{n}\setminus\{0\} \subseteq \Sigma_{\le h}(S)$.
\end{lemma}

\begin{proof}
Without loss of generality, let $|S|=n-1$ and
$$
S=g_1\cdots g_{n-1},
$$
where
$g_i\in C_n^\times$ for every $i\in[1,n-1]$.
Moreover, we can require that, if $g_i = g_j$ with $i\le j$, then $g_t = g_i$ for every $t\in [i, j]$.
It follows from $\mathsf{h}(S) \le h$ that
$g_i\neq g_{i+h}$ if $1\le i < i+h \le n-1$.

For every $j\in [1, h]\cap [1, n-1]$, let
$$
A_j=0\prod_{l=0}^{\lfloor \frac{n-1-j}{h} \rfloor}g_{j+lh}
$$
and $A_n=0$ if $h=n$.
Then $0^{h}S=\prod_{j=1}^{h}A_{j}$ and each $A_j$ is squarefree. By Corollary \ref{Corollary 2.5} and regarding a squarefree sequence as a set, we have
$$
\Big| \sum_{j=1}^{h}A_j \Big| \ge \min\{ n, \sum_{j=1}^h| A_j | -h+1\}=n.
$$
Therefore $C_n=\{0\}\cup \Sigma_{\le h}(S)$.
\end{proof}

We further require a structural dichotomy for sequences over $C_{nk}^{\times}$
with bounded multiplicity. This will provide the key alternative from which the
prime case can be completed.

\begin{lemma}\label{Lemma 3.6}
     For integers $n, k\ge 2$, let $S\in \mathcal{F}(C_{nk}^{\times})$ and $h\in[1, n-1]$. If $|S| \ge n $ and $\mathsf{h}(S) \le h$, then at least one of the following statements holds.

    {\rm (1)} There are two subsequences $T_1, T_2$ of $S$ such that $1\le |T_1| , |T_2| \le h$ and $\sigma(T_1)-\sigma(T_2) \in nC_{nk}\setminus \{0\}$;

    {\rm (2)} There is a subsequence $T$ of $S$ such that $1 \le |T| \le h$ and $\sigma(T)\in nC_{nk}\setminus \{0\}$.
\end{lemma}

\begin{proof}
    Assuming that {\rm (1)} does not hold, we will show that {\rm (2)} must then follow. Without loss of generality, let $| S |=n$ and
    $$S=g_1g_2\cdots g_{n}.$$
    Moreover, we can require that, if $g_i = g_j$ with $i\le j$, then $g_t = g_i$ for every $t\in [i, j]$.
    It follows from $\mathsf{h}(S) \le h$ that $g_i\neq g_{i+h}$ if $1\le i < i+h \le n$.

    For every $j\in[1, h]$, set
    $$
    B_j=0\prod_{l=0}^{\lfloor \frac{n-j}{h} \rfloor}g_{j+lh}.
    $$
    Then $0^{h}S=\prod_{j=1}^{h}B_{j}$ and each $B_j$ is squarefree. From Corollary \ref{Corollary 2.5}, we have
    \begin{equation*}
        \Big| \sum_{j=1}^{h}B_j \Big| \ge \min\{ nk, \sum_{j=1}^h| B_j | -h+1\}=n+1.
    \end{equation*}
    Thus, by the Pigeonhole Principle, there are two $h$-length sequences $U=0^{a}T_1$ and $V=0^{b}T_2$ with $a, b\in [0, h]$, $T_1, T_2 \mid S$
    and $\sigma(U)-\sigma(V) \in nC_{nk}\setminus \{0\}$.
    Since (1) does not hold, one of $U$ or $V$ must be of the form $0^h$ and the other one has sum in $nC_{nk}\setminus \{0\}$. So one of $T_1$ or $T_2$ has length in $[1, h]$ and sum in $nC_{nk}\setminus \{0\}$, which is as desired.
\end{proof}

We are now ready to determine the exact values of $\eta_{n}^{*}$, $\mathsf{D}_{n}^{*}$, $\eta_{n}^{\times}$ and $\mathsf{D}_n^{\times}$
when $n$ is a prime number.
Combining the preceding lemmas, we obtain the following result.

\begin{lemma}\label{Lemma 3.7}
    For any prime number $p$, we have $$\eta_{p}^{*}=\mathsf{D}_{p}^{*}=\eta_{p}^{\times}=\mathsf{D}_p^{\times}=2p-1.$$
\end{lemma}

\begin{proof}
    By Lemma \ref{Lemma 3.2} and Lemma \ref{Lemma 3.3}, it suffices to prove that every sequence
    $S\in \mathcal{F}(C_{p^2}^{\times})$ with $|S|=2p-1$ has a subsequence $T$ such that $ 1\le |T|\le p $ and $\sigma(T)\in pC_{p^2}\setminus\{0\}$.
    The case $p=2$ is trivial. So we may assume that $p$ is an odd prime. 

    If $\mathsf{h}(S)\ge p$, then $S$ has a subsequence $g^p$
    of length $p$ and sum $\sigma(g^p)=pg\in pC_{p^2}\setminus\{0\}$ as desired.
    If there are two elements $g', g''$ in $S$ satisfying $g'-g''\in pC_{p^2}\setminus \{0\}$, then the result follows from Lemma \ref{Lemma 3.4}.
    Therefore, we just need to deal with
    $$S=g_1^{\alpha_1}g_2^{\alpha_2}\cdots g_s^{\alpha_s} \in \mathcal{F}(C_{p^2}^{\times}),$$
    where $s=|\supp(S)|\in [3, p-1]$, $1\le \alpha_s \le \dots \le \alpha_1 \le p-1$, $|S|=\alpha_1+\dots+\alpha_s=2p-1$ and $pg_i\ne pg_j$ in $C_{p^2}$ for any $i\ne j$.

    Let $S_0=\emptyset$ be the empty sequence. For $j\in[1, \alpha_1]$, define
    \begin{equation*}
        S_j=\prod_{g\in \supp((\prod_{i=0}^{j-1}S_i)^{-1}S)}g
    \end{equation*}
    by induction on $j$.
    Then $S_j$ is squarefree for $j\in[1, \alpha_1]$ and $S=\prod_{j=1}^{\alpha_1}S_j$. Furthermore,
    $$g_1 \mid S_{\alpha_1} \mid S_{\alpha_1-1} \mid \cdots
    \mid S_2 \mid S_1 = g_1\cdot\ldots\cdot g_s$$
    and so $1\le |S_{\alpha_1}|\le \cdots \le |S_1|=s \le p-1$.
    Since $|S|=\sum_{j=1}^{\alpha_1}|S_j|=2p-1$, there is an integer $k \in [2, \alpha_1]$ such that
    \begin{equation*}
        \sum_{j=1}^{k-1}|S_j|< p \le \sum_{j=1}^{k}|S_j|.
    \end{equation*}
    It must hold that $k<\alpha_1$. Otherwise, $k = \alpha_1$ leads to
   \begin{equation*}
       \sum_{j=1}^{k-1}|S_j|=\sum_{j=1}^{\alpha_1-1}|S_j|
       =|S|-|S_{\alpha_1}| \ge (2p-1)-(p-1)=p,
   \end{equation*}
   a contradiction to $\sum_{j=1}^{k-1}|S_j|< p$.

   Set $m=\sum_{j=1}^{k}|S_j| - p$. Then $s-m=|S_1|-m=p-\sum_{j=2}^{k}|S_j| \ge 1$.
   Now we define
   \begin{equation*}
       W=\prod_{i=1}^{s-m}g_i \prod_{j=2}^{k}S_j.
   \end{equation*}
    Then $\mathsf{h}(W)=k$ and
    $$
    |W|=(s-m)+\sum_{j=2}^k|S_j|=\sum_{j=1}^k|S_j|-m=p.
    $$
    If there is a subsequence $T$ of $W$ such that $1\le |T| \le k$ and $\sigma(T)\in pC_{p^2}\setminus \{0\}$, then $T$ is as desired and we are done.
    If such a subsequence of $W$ does not exist, then by Lemma \ref{Lemma 3.6},
    there are two subsequences $W_1, W_2$ of $W$ such that $1\le | W_1| , | W_2| \le k$ and $\sigma(W_1)-\sigma(W_2) \in pC_{p^2}\setminus \{0\}$.

    Let $\pi:C_{p^2}\to C_{p^2}/pC_{p^2}\cong C_p$
    be the canonical homomorphism. If $\sigma(\pi(W_1))=\pi(0)=0+pC_{p^2}$,
    then $\sigma(W_1)\in pC_{p^2}$. Since
    $\sigma(W_1)-\sigma(W_2)\in pC_{p^2}\setminus\{0\},$
    at least one of $\sigma(W_1)$ and $\sigma(W_2)$ belongs to
    $pC_{p^2}\setminus\{0\}$, and we are done. Hence we may assume that
    $\sigma(\pi(W_1))\neq 0+pC_{p^2}$.

   Set
   \begin{equation*}
       U=W^{-1}S=\prod_{i=s-m+1}^{s}g_i\prod_{j=k+1}^{\alpha_1}S_i.
   \end{equation*}
   Then $|\pi(U)|=|U|=|S|-|W|=p-1$ and
   $\mathsf{h}(\pi(U))=\mathsf{h}(U)\le \alpha_1-k+1\le p-k$.
   It follows from Lemma \ref{Lemma 3.5} that
   $(C_{p^2}/pC_{p^2})^{\times} \subseteq \Sigma_{\le p-k}(\pi (U))$.
   Thus, there exists a subsequence $U'\mid U$ satisfying $1\le |U'| \le p-k$ and $\sigma(\pi(U'))=-\sigma(\pi(W_1))\in (C_{p^2}/pC_{p^2})^{\times}$.
   Hence, $\sigma(\pi(U'W_1))=0+pC_{p^2}\in C_{p^2}/pC_{p^2}$,
   or equivalently, $\sigma(U'W_1)\in ker(\pi)=pC_{p^2}$.

   Let $T_1=U'W_1,\ T_2=U'W_2$.
   Since $\sigma(W_1)-\sigma(W_2) \in pC_{p^2}\setminus \{0\}$, $S$ has two subsequences $T_1, T_2$ such that $1\le|T_1|, |T_2| \le p$ and $\sigma(T_1), \sigma(T_2)\in pC_{p^2}, \sigma(T_1)-\sigma(T_2) \in pC_{p^2}\setminus\{0\}$.
   Then either $T_1$ or $T_2$ has length in $[1, p]$ and sum in $pC_{p^2}\setminus \{0\}$, which is as desired.
\end{proof}

Lemma \ref{Lemma 3.7} provides the exact values of the four combinatorial constants  $\eta_{n}^{*}$, $\mathsf{D}_{n}^{*}$, $\eta_{n}^{\times}$ and $\mathsf{D}_n^{\times}$ when $n$ is prime. In the next
section, we extend this result to arbitrary integers $n$ by means of
inductive arguments and recursive decompositions.

\section{Proofs of Theorem \ref{Theorem 1.2} and Theorem \ref{Theorem 1.4} }

In this section, we give the proofs of Theorem \ref{Theorem 1.2} and Theorem \ref{Theorem 1.4}.
First, we prove Theorem \ref{Theorem 1.4}, or equivalently, $\eta_n^{\times}\le 2n-1$ for a prime power $n$.

\

\noindent {\bf  Proof of Theorem \ref{Theorem 1.4}. }
    Let $n=p^e$ with a positive integer $e$ and let $G=C_{p^{2e}}$. By Proposition \ref{Proposition 3.1}, we will prove that every sequence $S\in \mathcal{F}(G\setminus pG)$ with $|S|=2p^e-1$ has a subsequence $T$ with $ 1\le |T|\le p^e$ and $\sigma(T)\in p^eG\setminus p^{e+1}G.$

    We claim that, for every $j\in[0,e]$, $S$ has disjoint subsequences  $T_1^{(j)},T_2^{(j)},\ldots,T_{s_j}^{(j)}$ with $s_j=2p^{e-j}-1$, $1\le |T_i^{(j)}|\le p^j$ and $\sigma(T_i^{(j)})\in p^jG\setminus p^{j+1}G$ for $i\in[1,s_j]$.
    We will prove our claim by induction on $j$. Our claim for $j=0$ is clear by the hypothesis at the first paragraph.
    Assume that our claim holds for $j-1$ with $j\in[1,e]$.

    By the induction assumption,
    $S$ has disjoint subsequences  $T_1^{(j-1)},T_2^{(j-1)},\ldots,T_{s_{j-1}}^{(j-1)}$ with $s_{j-1}=2p^{e-j+1}-1$, $1\le |T_i^{(j-1)}|\le p^{j-1}$ and $\sigma(T_i^{(j-1)})\in p^{j-1}G\setminus p^{j}G$ for $i\in[1,s_{j-1}]$. Set
    $$
    S^{j-1}=\sigma(T_1^{(j-1)})\sigma(T_2^{(j-1)})\cdots \sigma(T_{s_{j-1}}^{(j-1)})
    \in \mathcal{F}(p^{j-1}G\setminus p^{j}G).
    $$
    By $\eta_p^{\times}=2p-1\le 2p^{e-j+1}-1=|S^{j-1}|$, we can choose disjoint subsequences $\prod _{i\in I_1}\sigma(T_i^{(j-1)}),\ldots,\prod _{i\in I_w}\sigma(T_i^{(j-1)})$ of $S^{j-1}$ with
    $$
    w\ge
    \left\lceil \frac{|S^{j-1}|-(2p-2)}{p}\right\rceil= \left\lceil \frac{2p^{e-j+1}-1-(2p-2)}{p}\right\rceil=2p^{e-j}-1=s_j
    $$
    such that $\bigcup_{i=1}^wI_i\subseteq[1,s_{j-1}]$, $|I_x|\in[1,p]$ and
    $$
    \sigma(\prod _{i\in I_x}\sigma(T_i^{(j-1)}))\in p^jG\setminus p^{j+1}G,
    $$
    for any $x\in[1,w]$.

    For any $x\in[1,w]$, let $T_x^{(j)}=\prod _{i\in I_x}T_i^{(j-1)}$. Then $1\le|T_x^{(j)}|=\sum_{i\in I_x}|T_i^{(j-1)}|\le p^j$ and
    $$
    \sigma(T_x^{(j)})=\sigma(\prod _{i\in I_x}\sigma(T_i^{(j-1)}))\in p^jG\setminus p^{j+1}G.
    $$
    Therefore our claim is proved for $j$.

    Theorem \ref{Theorem 1.4} follows from our claim for $j=e$.
    \hfill$\Box$

Next, we are going to prove Theorem \ref{Theorem 1.2} and so turn to $\mathsf{D}^*(n,nk)$, which is equal to $\eta^*(n,nk)$ by Lemma \ref{Lemma 3.2}.

\begin{lemma}\label{Lemma 4.1}
    Let $m,n,k \ge 2$ be integers.
    If $\mathsf{D}^*(n,nk)<\infty$ and $\mathsf{D}^*(m,mk)<\infty$,
    then we have
    $$
    \mathsf{D}^*(mn,mnk)\le
    n\bigl(\mathsf{D}^*(m,mk)-1\bigr)+\mathsf{D}^*(n,nk).
    $$
\end{lemma}

\begin{proof}
    Let $G=C_{nmk}$ and $S\in\mathcal F(G \setminus nmG)$
    be a sequence of length
    \begin{equation}\label{4-4}
        |S|=n\bigl(\mathsf{D}^*(m,mk)-1\bigr)+\mathsf{D}^*(n,nk).
    \end{equation}
    We will prove that $S$ has a subsequence $T$ with $\sigma(T)\in nmG\setminus nmkG=nmG\setminus\{0\}$.

    Since $S\in \mathcal{F}(G \setminus nmG)$, we decompose $S$ as
    $$S = S_1 S_2,$$
    where every term in $S_1$ belongs to $G \setminus nG$ and every term in $S_2$ belongs to $nG \setminus mnG$.

    Let $\varphi:G\to G/nkG$ be the canonical homomorphism.
    Then
    \begin{equation*}
        \varphi(S_1)\in  \mathcal{F}( \varphi(G) \setminus \varphi(nG)),
    \end{equation*}
    where $\varphi(G)= G/nkG \cong C_{nk}$ and $\varphi(nG)= nG/nkG \cong C_{k}$.
    If $|S_1|<\mathsf{D}^*(n,nk)$, then set $r=0$.
    If $|S_1|\ge \mathsf{D}^*(n,nk)$, then by $\eta^*(n,nk)=\mathsf{D}^*(n,nk)$ and using the same method in the proof of Theorem \ref{Theorem 1.4},
    we can choose disjoint subsequences $T_1, \ldots, T_r$ of $S_1$ with
    \begin{equation}\label{4-5}
        r\ge\left\lceil \frac{| S_1|-(\mathsf{D}^*(n,nk)-1)}{n}\right\rceil
    \end{equation}
    such that $|T_i|\in[1, n]$ and
    \begin{equation}\label{4-6}
        \varphi(\sigma(T_i))=\sigma(\varphi(T_i))\in \varphi(nG) \setminus \varphi(nkG)
    \end{equation}
    for every $i\in[1,r]$.
    If there exists an index $i \in [1, r]$ satisfying $\sigma(T_i) \in nmG$, then it follows from (\ref{4-6}) that $\sigma(T_i) \in nmG \setminus \{0\}$
    and so $T_i$ is as desired.
    If not, then $\sigma(T_i) \in nG \setminus nmG$ for every $i\in[1,r]$.

    Now consider the sequence
    \begin{equation*}
        W=S_2 \cdot \prod_{i=1}^{r} \sigma(T_i)\in\mathcal{F}(nG\setminus mnG)
    \end{equation*}
    with $W=S_2$ if $r=0$.
    Using (\ref{4-4}) and (\ref{4-5}), we have
    \begin{equation*}
        \begin{aligned}
        |W|
        &= r+|S_2| \\
        &\ge \left\lceil \frac{|S_1|-(\mathsf{D}^*(n,nk)-1)}{n}\right\rceil + |S_2| \\
        &\ge \left\lceil \mathsf{D}^*(m,mk)-1+\frac{1-|S_2|}{n}\right\rceil+|S_2| \\
        &\ge \mathsf{D}^*(m,mk).
        \end{aligned}
    \end{equation*}
    By the definition of $\mathsf{D}^*(m,mk)$ and $W \in \mathcal{F}(nG \setminus mnG)$, there exists a non-empty subsequence
    $$
    W'=S_2' \cdot \prod_{i\in I}\sigma(T_i)
    $$
    of $W$ such that $S_2'\mid S_2$, $I\subseteq [1,r]$
    and $\sigma(W')\in mnG\setminus mknG$.

    Finally, let
    \begin{equation*}
         S_3=S_2' \cdot \prod_{i\in I}T_i .
    \end{equation*}
    Then $S_3\mid S$ and
    \begin{equation*}
         \sigma(S_3)=\sigma(S_2')+\sum_{i\in I} \sigma(T_i)
         =\sigma(W')\in nmG\setminus\{0\}.
    \end{equation*}
    Thus, $S_3$ is the required subsequence of $S$, and the proof is complete.
\end{proof}

Now, we can prove Theorem \ref{Theorem 1.2}.

\

\noindent {\bf Proof of Theorem \ref{Theorem 1.2}.}
{\rm (1)}
By the hypothesis $\rad(n) \nmid \rad(k)$,
there exists a prime $p$ such that $p \mid n$ and $p \nmid k$.
Let $nk=mp^v$ with $p \nmid m$ and $p^v \mid n$.

Consider the sequence $S$ consisting entirely of the integer $m$,
which is not divisible by $n$.
Then any subsequence $T$ of $S$ has sum $\sigma(T)=m|T|$.
If $\sigma(T)$ is divisible by $n$,
then it follows from $p^v \mid n$ and $p \nmid m$ that $p^v \mid |T|$,
which implies $nk=mp^v \mid m|T|=\sigma(T)$.
Therefore, no matter how long $S$ is, $S$ has no subsequence whose sum is divisible by $n$ but not by $nk$.

This yields $\eta^*(n,nk)=\mathsf{D}^*(n,nk)=\infty$.

{\rm (2)}
By Lemma \ref{Lemma 4.1}, if $\mathsf{D}^*(a,ak)\le 2a-1$ and $\mathsf{D}^*(b,bk)\le 2b-1$, then we have
\begin{equation*}
\begin{aligned}
\mathsf{D}^*(ab, abk)
&\le a\bigl(\mathsf{D}^*(b,bk)-1\bigr)+\mathsf{D}^*(a,ak) \\
&\le a(2b-2)+2a-1=2ab-1.
\end{aligned}
\end{equation*}
It follows from Definition \ref{Definition 1.1} and Lemma \ref{Lemma 3.7} that $\mathsf{D}^*(p,pk)\le \mathsf{D}^*(p,p^2)=2p-1$ for any prime $p \mid k$.
Since every prime divisor of $n$ is a prime divisor of $k$ by $1<\rad(n)\mid \rad(k)$, we get $\mathsf{D}^*(n,nk)\le 2n-1$ as desired.
\hfill$\Box$

\section{Inverse problems}
In this section, we study the inverse problems of Theorem \ref{Theorem 1.2} and Theorem \ref{Theorem 1.4}.
We first solve the inverse problem of Theorem \ref{Theorem 1.2}.

\begin{lemma}\label{Lemma 5.1}
    Let $n, k, m$ be positive integers with $m \mid k$, $\gcd(m, k/m)=1$ and $\rad(k/m)=\rad(n)>1$.
    For $G=C_{nk}$, if $S\in \mathcal{F}(G \setminus nG)$ of length $|S|=2n-2$
    and $S$ has no subsequence $T$ of length $|T|\in [1, n]$ satisfying $\sigma(T)\in nG\setminus\{0\}$,
    then $mg \in (mG)^{\times}$ for any $g \mid S$.
\end{lemma}

\begin{proof}
    Let $G=\langle e \rangle$ and $ue \mid S$ with $u\in [0, nk-1]$.
    The hypothesis $S\in \mathcal{F}(G \setminus nG)$ yields $n\nmid u$.

    Assume $mue \not\in (mG)^{\times}$.
    Then together with $mG=\langle me \rangle$,
    it follows that $\gcd(u, nk/m)=\gcd(u, |mG|)>1$.
    So $\gcd(u, \rad(n))=\gcd(u, \rad(nk/m))>1$ by the hypothesis $\rad(k/m)=\rad(n)$.
    Choose a prime divisor $p \mid \gcd(u,n)$ and let $n=pw$, where $w\ge 2$ by $n\nmid u$.

    Since $S\in \mathcal{F}(G \setminus nG)$, we can decompose $S$ as
    $$S = S_1 S_2,$$
    where $S_1\in \mathcal{F}(G \setminus pG)$ and $S_2\in \mathcal{F}(pG \setminus nG)$.
    It follows from $p \mid u$ that $ue\mid S_2$ and so $|S_2|\ge 1$.
    Using the same argument as in Lemma \ref {Lemma 4.1}, we have
    \begin{equation*}
        \begin{aligned}
        \left\lceil \frac{|S_1|-\eta^*(p, pwk)+1}{p}\right\rceil + |S_2|
        &=\left\lceil\frac{(2pw-2-|S_2|)-2p+2}{p}\right\rceil+|S_2|\\
        &=2w-2+\left\lceil \frac{-|S_2|}{p}\right\rceil+|S_2|\\
        &\ge 2w-1 =\eta^*(w,wk),
        \end{aligned}
    \end{equation*}
    and so $S$ has a subsequence $T$ of length $|T|\in [1, n]$
    with $\sigma(T)\in wpG\setminus wkpG=nG\setminus\{0\}$,
    a contradiction to the hypothesis of Lemma \ref {Lemma 5.1}.
\end{proof}

\begin{lemma}\label{Lemma 5.3}
    Let $S\in \mathcal{F}(C_n^{\times})$ with $n\ge 3$
    be a sequence of length $|S|=\ell \in [1, n-2]$.
    If $|\Sigma(0S)|=\ell+1$, then $\supp(S)\subseteq \{g,-g\}$ for some $g\in C_n^{\times}$.
\end{lemma}

\begin{proof}
    Let $S=g_1 \cdots  g_{\ell}$, $S_j=g_1\cdots g_j$ and
    $A_j=\Sigma(0S_j)=\{0\}\cup \Sigma(S_j)$ for every $j\in[1,\ell]$.
    Then for any $j\in[2,\ell]$, we have
    $$
    A_j=A_{j-1}+\{0,g_j\}.
    $$
    Since $g_j\in C_n^{\times}$ and $|A_j|\le |A_\ell|=|\Sigma(0S)|=\ell+1<n$ by $\ell \in [1, n-2]$,
    Lemma \ref{Lemma 2.4} implies that
    $$
    |A_j|\ge \min\{n,\ |A_{j-1}|+2-1\}=|A_{j-1}|+1.
    $$
    Then $2=|A_1|< \cdots <|A_\ell|=\ell+1$ and so $|A_j|=j+1$.

    Let $g=g_1\in C_n^{\times}$ and suppose that $S_{j-1} \in \mathcal{F}(\{g,-g\})$ for some $j \in [2, \ell]$.
    Then $A_{j-1}=\Sigma(0S_{j-1})$ is an arithmetic progression in $G$ with common difference $g\in C_n^{\times}$ and cardinality $|A_{j-1}|=j$.
    Since
    $$
    |A_{j-1}+\{0,g_j\}|=|A_j|=j+1=|A_{j-1}|+|\{0,g_j\}|-1,
    $$
    Lemma \ref{Lemma 2.6} yields that $\{0,g_j\}$ is an arithmetic progression with the same common difference $g$, or equivalently, $g_j \in \{g,-g\}$.
    Therefore, by induction on $j$, we get $S=S_{\ell} \in \mathcal{F}(\{g,-g\})$.
\end{proof}

\begin{lemma}\label{Lemma 5.4}
    Let $n, k$ be positive integers with $n\ge 3$ and $\rad(k)=\rad(n)$.
    For $G=C_{nk}$, if $S\in \mathcal{F}(G \setminus nG)$ of length $|S|=2n-2$
    and $S$ has no subsequence $T$ of length $|T|\in [1, n]$ satisfying $\sigma(T)\in nG\setminus\{0\}$,
    then
    $$
    S=g^{n-1}(-g)^{n-1}
    $$
    for some $g \in G^{\times}$.
\end{lemma}

\begin{proof}
By Lemma \ref{Lemma 5.1} and $\rad(k)=\rad(n)$,
we have $S\in\mathcal{F}(G^{\times})$.
It is clear that $\mathsf{h}(S)\le n-1$.

Let $\pi:G\rightarrow G/nG\cong C_n$ be the canonical homomorphism.
Since $|S|=2n-2\ge n+1$ and $S$ has no subsequence $T$ of length $|T|\in [1, n]$ satisfying $\sigma(T)\in nG\setminus\{0\}$, Lemma \ref{Lemma 3.4} implies that  $x=y$ if $\pi(x)=\pi(y)$ and $xy\mid S$.
Thus $\mathsf{h}(\pi(S))=\mathsf{h}(S)\le n-1$.

Fix an element $g\mid S$. For arbitrary $g_1 \cdots g_{n-2} \mid g^{-1}S$,
let $S_1=gg_1 \cdots g_{n-2}$ and $S_2=S_1^{-1}S$.
Since $\pi(S_i)\in \mathcal{F}((G/nG)^{\times})$, $|\pi(S_i)|=n-1$
and $\mathsf{h}(\pi(S_i))\le \mathsf{h}(\pi(S))\le n-1$ for each $i\in [1, 2]$,
it follows from Lemma \ref{Lemma 3.5} that
\begin{equation}\label{5-1}
\Sigma(\pi(0S_1))=\Sigma(\pi(0S_2))=G/nG.
\end{equation}

Let $U$ and $V$ be any two subsequences of $0S_1$ with $\sigma(\pi(U))=\sigma(\pi(V))$.
By (\ref{5-1}), there exists a non-empty subsequence $W\mid 0S_2$ such that $\sigma(\pi(W))=-\sigma(\pi(U)).$
Together with $\sigma(\pi(U))=\sigma(\pi(V)),$
we have $\sigma(\pi(UW))=\sigma(\pi(VW))=\pi(0).$
By the hypothesis of Lemma \ref{Lemma 5.4}
and the same argument as in Lemma \ref{Lemma 3.2}{\rm (1)},
$S$ has no subsequence $T$ satisfying $\sigma(T)\in nG\setminus\{0\}$.
It follows that $\sigma(UW)=\sigma(VW)=0$ as well as $\sigma(U)=\sigma(V)$, whether $0\mid R$ for $R\in\{U,V,W\}$ or not.
That is, the value of $\sigma(T)$ with $T\mid S_1$ is uniquely determined by $\sigma(\pi(T))$.
Therefore,
\begin{equation}\label{5-2}
|\Sigma(0S_1)|=|\Sigma(\pi(0S_1))|=n
\end{equation}
by (\ref{5-1}).

Since $S_1\in\mathcal{F}(C_{nk}^{\times})$ of length $|S_1|=n-1<nk-2$ and (\ref{5-2}), Lemma \ref{Lemma 5.3} yields $\supp(S_1)\subseteq\{g,-g\}$.
By the arbitrariness of $g_1 \cdots g_{n-2} \mid g^{-1}S$,
we get $\supp(S)\subseteq\{g,-g\}$.
Together with $\mathsf{h}(S)\le n-1$ and $|S|=2n-2$, it forces that
$$
S=g^{n-1}(-g)^{n-1}.
$$
\end{proof}

\begin{lemma}\label{Lemma 5.5}
    Let $n, k, m$ be positive integers with $m \mid k$, $\gcd(m, k/m)=1$ and $\rad(k/m)=\rad(n)>1$. Let $G=C_{nk}$ and $S\in \mathcal{F}(G \setminus nG)$ of length $|S|=2n-2$.
    If $S$ has no subsequence $T$ of length $|T|\in [1, n]$ satisfying $\sigma(T)\in nG\setminus\{0\}$,
    then
    $$
    S=g^{n-1}(-g)^{n-1},
    $$
    where $g=g_1+g_2$ for some $g_1\in\frac{nk}{m}G$ and $g_2 \in (mG)^{\times}$.
\end{lemma}

\begin{proof}
The case $n=2$ is trivial and we assume that $n\ge 3$. Let $G=\langle e_0 \rangle$. Since $\gcd(m, \frac{nk}{m})=1$, we have $G=G_1 \oplus G_2$,
where $G_1=\frac{nk}{m}G=\langle e_1 \rangle \cong C_{m}$ with $e_1=\frac{nk}{m}e_0$
and $G_2=mG=\langle e_2 \rangle \cong C_{nk/m}$ with $e_2=me_0$.

Define $\varphi(x)=mx$ for any $x\in G$.
Then by $S\in \mathcal{F}(G \setminus nG)$ and Lemma \ref{Lemma 5.1}, we have
$$
\varphi(S) \in \mathcal{F}(G_2^{\times} \setminus nG_2).
$$
Moreover, the hypothesis of Lemma \ref{Lemma 5.5} yields $\gcd(m, n)=1$,
and thus $\varphi(S)$ has no subsequence $\varphi(T')$ of length $|\varphi(T')|\in [1, n]$ satisfying $\sigma(\varphi(T'))\in mnG\setminus\{0\}=nG_2\setminus\{0\}$.
Together with $\rad(k/m)=\rad(n)$ and $|\varphi(S)|=|S|=2n-2$,
Lemma \ref{Lemma 5.4} implies that
$$
\varphi(S)=(g_3)^{n-1}(-g_3)^{n-1}
$$
for some $g_3\in G_2^{\times}$.

For any $g\in G$, $g$ has a unique expression as $g=xe_1+ye_2$
with integers $x\in [0, m-1]$ and $y\in [0, nk/m-1]$.
If $\varphi(g)=mye_2=g_3$, then it follows from $e_2,g_3\in G_2^{\times}$ that there is a unique $g_2\in G_2^{\times}$ satisfying $ye_2=g_2$. If $\varphi(g)=-g_3$, then $\varphi(-g)=g_3$ and so $-ye_2=g_2$.
Therefore,
$$
S=\prod_{i=1}^{n-1}(a_ie_1+g_2)\prod_{i=1}^{n-1}(b_ie_1-g_2)
$$
with $a_i, b_i\in [0, m-1]$ for $i\in [1, n-1]$.

Consider the subsequence $R=(a_ie_1+g_2)(b_je_1-g_2) \mid S$
with $|R|=2<n$ and $\sigma(R)=(a_i+b_j)e_1\in nG$.
By the hypothesis of Lemma \ref{Lemma 5.5},
it must hold that $\sigma(R)=(a_i+b_j)e_1=0$.
Therefore, $a_i+b_j\equiv 0 \pmod{m}$ for any $i, j\in [1, n-1]$. By the arbitrariness of $a_i$ and $b_j$, the sequence $S$ must be of the form
$$
S=(g_1+g_2)^{n-1}(-g_1-g_2)^{n-1},
$$
where $g_1\in G_1=\frac{nk}{m}G$.
\end{proof}

Now we can prove Theorem \ref{Theorem 1.3}.

\

\noindent {\bf Proof of Theorem \ref{Theorem 1.3}.}
Let
$$
\mathbb{Z} \xrightarrow{\pi} \mathbb{Z}/nk\mathbb{Z},
$$
where $\pi$ is the canonical homomorphism.
Using Lemma \ref{Lemma 5.5} for $G=(\mathbb{Z}/nk\mathbb{Z}, +)\cong C_{nk}$, we have
$$
\pi(S)=(g_1+g_2)^{n-1}(-g_1-g_2)^{n-1}
\in \mathcal{F}(\mathbb{Z}/nk\mathbb{Z}),
$$
where $g_1\in\frac{nk}{m}G=\frac{nk}{m}\mathbb{Z}/nk\mathbb{Z}$ and $g_2 \in (mG)^{\times}=(m\mathbb{Z}/nk\mathbb{Z})^{\times}$.
Therefore, there exist integers $c,d$ with $c$ coprime to $\frac{nk}{m}$, as well as $n$ by $\rad(k/m)=\rad(n)$, such that $n-1$ terms of $S$ are congruent to $d\frac{nk}{m}+cm$ modulo $nk$ and the other $n-1$ terms of $S$ are congruent to $-d\frac{nk}{m}-cm$ modulo $nk$.

It is clear that such a sequence $S$ has no subsequence whose sum is divisible by $n$ but not by $nk$.
\hfill$\Box$

Next, we turn to the inverse problem of Theorem \ref{Theorem 1.5}.

\begin{lemma}\label{Lemma 5.6}
     Let $G=C_{p^{e+1}}$ with $p$ prime and $e\ge 1$,
     and $S\in \mathcal{F}(G^{\times})$ of length $|S|=2p^e-2$.
     If $S$ has no subsequence $T$ of length $|T|\in [1, p^{e}]$ satisfying $\sigma(T)\in (p^e G)^{\times}$, then
     $$
     S=g^{p^e-1}(-g)^{p^e-1},
     $$
     for some $g\in G^{\times}$.
\end{lemma}

\begin{proof}
The case $p^e=2$ is trivial and we assume that $p^e \ge 3$.
Note that $S\in \mathcal{F}(G^{\times}) \subset \mathcal{F}(G \setminus p^e G)$
and $(p^e G)^{\times}=p^e G \setminus \{0\}$ since $G=C_{p^{e+1}}$.
This lemma immediately follows from Lemma \ref{Lemma 5.4}
by taking $n=p^e$ and $k=p$.
\end{proof}

\noindent {\bf Proof of Theorem \ref{Theorem 1.5}.}
Let
$$
\mathbb{Z} \xrightarrow{\pi} \mathbb{Z}/np\mathbb{Z}
$$
where $\pi$ is the canonical homomorphism.
Noting that $n$ is a power of a prime $p$ and using Lemma \ref{Lemma 5.6} for $G=(\mathbb{Z}/np\mathbb{Z}, +)\cong C_{np}$, we have
$$
\pi(S)=g^{n-1}(-g)^{n-1}
\in \mathcal{F}(\mathbb{Z}/np\mathbb{Z}),
$$
where $g \in G^{\times}=(\mathbb{Z}/np\mathbb{Z})^{\times}$.

Let $g=\pi(c)=c+np\mathbb{Z}$ for some $c \in \mathbb{Z}$. Then $c$ is an integer coprime to $np$.
Therefore, $n-1$ terms of $S$ are congruent to $c$ modulo $np$ and the other $n-1$ terms of $S$ are congruent to $-c$ modulo $np$.

It is clear that such a sequence $S$ has no subsequence whose sum $\sigma$ satisfies $\gcd(\sigma, n^2)=n$.
\hfill$\Box$

\section{Concluding remarks}

Theorem \ref{Theorem 1.4} states that $\eta^{\times}_{n}=\mathsf{D}^{\times}_{n}=2n-1$ for a prime power $n$.
It is natural to consider the following problem.

\begin{problem}
    What are the values of $\eta_{n}^{\times}$  and $\mathsf{D}^{\times}_{n}$ for an integer $n$ with $\omega(n)\ge2$?
\end{problem}

Together with Theorem \ref{Theorem 1.4}, the following example implies that
$\eta_{n}^{\times}\ge \mathsf{D}_{n}^{\times}\ge2n+2\omega(n)-3$
for any integer $n\ge 2$.

\begin{example}\label{Example 6.2}
    Let $n$ be a positive integer with $\omega(n)=w\ge 2$
    and $\rad(n)=q_1\cdots q_w$ with distinct prime $q_1<\ldots<q_w$.
    Over the cyclic group $G=C_{n^2}=\langle g\rangle $, consider the sequence
    $$
    S=g^{n-1}(-g)^{n-1}((1+xn)g)^{w-1}((-1-xn)g)^{w-1},
    $$
    where the integer $x$ satisfies
    \begin{equation*}
    \left\{
    \begin{aligned}
            x & \equiv 0 \pmod{q_1}, \\
          x+1 & \equiv 0 \pmod{q_2}, \\
         2x+1 & \equiv 0 \pmod{q_3}, \\
              & \cdots \\
     (w-1)x+1 & \equiv 0 \pmod{q_w}.
    \end{aligned}
    \right.
    \end{equation*}
    Then $S\in\mathcal{F}(G^{\times})$ is of length $|S|=2n+2w-4$
    and $S$ has no subsequence $T$ such that $\sigma(T)\in (nG)^{\times}$.
\end{example}

By the Chinese Remainder Theorem and $q_i> i-1$ for $i\in[1,w]$, the congruences in Example \ref{Example 6.2} have a unique solution $x$ modulo $\rad(n)=q_1\cdots q_w$.
Thus, the sequence $S$ in Example \ref{Example 6.2} must exist for any integer $n$ with $\omega(n)\ge 2$.

Let
$$
T=g^{a_1}(-g)^{b_1}((1+xn)g)^{a_2}((-1-xn)g)^{b_2} \mid S
$$
with $\sigma(T)\in nG$.
It follows that
$$
a_1+a_2-b_1-b_2 \equiv 0 \pmod{n}.
$$
Since $0\le a_1, b_1 \le n-1$ and $0\le a_2, b_2 \le \omega(n)-1<n$,
we have
$$
a_1+a_2-b_1-b_2 \in \{-n, 0, n\}.
$$
If $a_1+a_2-b_1-b_2=0$, then
$$
\sigma(T)=(a_1+a_2-b_1-b_2+(a_2-b_2)xn)g=(a_2-b_2)xng \not\in (nG)^{\times}
$$
since $q_1 \mid \gcd(x, n)=\gcd(x, \exp(nG))$.
If $a_1+a_2-b_1-b_2=n$, then
$$
1\le n-a_1+b_1 \le a_2-b_2 \le w-1
$$
and
$$
\sigma(T)=((a_2-b_2)x+1)ng \not\in (nG)^{\times}
$$
since $\gcd((a_2-b_2)x+1, n)>1$ by the congruences in Example \ref{Example 6.2}.
If $a_1+a_2-b_1-b_2=-n$, then
$$
1\le n-b_1+a_1=b_2-a_2\le w-1
$$
and
$$
\sigma(T)=-((b_2-a_2)x+1)ng \not\in (nG)^{\times}
$$
since $\gcd((b_2-a_2)x+1, n)>1$ by the congruences in Example \ref{Example 6.2}.
In conclusion, it holds that $\sigma(T) \not\in (nG)^{\times}$.

\end{document}